\documentclass[reqno,10pt,centertags]{amsart}
\usepackage[letterpaper,margin=1.5in,bottom=1.3in]{geometry}
\usepackage{amsmath,amsthm,amscd,amssymb,latexsym,enumerate}

\usepackage{hyperref}

\newcommand*{\mailto}[1]{\href{mailto:#1}{\nolinkurl{#1}}}
\newcommand{\arxiv}[1]{\href{http://arxiv.org/abs/#1}{arXiv: #1}}

\makeatletter
\def\theequation{\@arabic\c@equation}

\newcommand{\bbN}{{\mathbb{N}}}
\newcommand{\bbR}{{\mathbb{R}}}

\newcommand{\cB}{{\mathcal B}}

\newcommand{\cH}{{\mathcal H}}

\newcommand{\no}{\nonumber}
\newcommand{\lb}{\label}
\newcommand{\bi}{\bibitem}
\newcommand{\f}{\frac}

\newcommand{\ol}{\overline}

\newcommand{\wti}{\widetilde}

\newcommand{\dom}{\operatorname{dom}}

\newcommand{\supp}{\operatorname{supp}}

\renewcommand{\ln}{\operatorname{ln}}

\numberwithin{equation}{section}

\newtheorem{theorem}{Theorem}[section]

\newtheorem{corollary}[theorem]{Corollary}

\theoremstyle{definition}

\newtheorem{remark}[theorem]{Remark}

\begin{document}

\title[Hardy--Rellich-Type Inequalities]{Factorizations and Hardy--Rellich-Type Inequalities}

\author[F.\ Gesztesy]{Fritz Gesztesy}
\address{Department of Mathematics,
Baylor University, One Bear Place \#97328,
Waco, TX 76798-7328, USA}
\email{Fritz\_Gesztesy@baylor.edu}
\urladdr{http://www.baylor.edu/math/index.php?id=935340}

\author[L.\ L.\ Littlejohn]{Lance Littlejohn}
\address{Department of Mathematics,
Baylor University, One Bear Place \#97328,
Waco, TX 76798-7328, USA}
\email{Lance\_Littlejohn@baylor.edu}
\urladdr{http://www.baylor.edu/math/index.php?id=53980}


\dedicatory{Dedicated with great pleasure to Helge Holden on the occasion of his 
60th birthday.} 
\date{\today}
\thanks{To appear in {\it Partial Differential Equations, Mathematical Physics, and Stochastic Analysis.  A Volume in Honor of Helge Holden's 60th Birthday}, 
EMS Congress Reports, F.\ Gesztesy, H.\ Hanche-Olsen, E.\ Jakobsen, Y.\ Lyubarskii, N.\ Risebro, and K.\ Seip (eds.).}
\subjclass[2010]{Primary: 35A23, 35J30; Secondary: 47A63, 47F05.}
\keywords{Hardy inequality, Rellich-type inequality, factorizations of differential operators.}

\begin{abstract}
The principal aim of this note is to illustrate how factorizations of singular, even-order partial differential operators yield an elementary approach to classical inequalities of Hardy--Rellich-type. More precisly, introducing the two-parameter $n$-dimensional homogeneous scalar differential expressions $T_{\alpha,\beta} := - \Delta + \alpha |x|^{-2} x \cdot \nabla 
+ \beta |x|^{-2}$, $\alpha, \beta \in \bbR$, $x \in \bbR^n \backslash \{0\}$, $n \in \bbN$, 
$n \geq 2$, and its formal adjoint, denoted by $T_{\alpha,\beta}^+$, we show that nonnegativity of $T_{\alpha,\beta}^+ T_{\alpha,\beta}$ on $C_0^{\infty}(\bbR^n \backslash \{0\})$ implies the fundamental inequality, 
\begin{align}
\int_{\bbR^n} [(\Delta f)(x)]^2 \, d^n x 
& \geq [(n - 4) \alpha - 2 \beta] \int_{\bbR^n} |x|^{-2} |(\nabla f)(x)|^2 \, d^n x  \no \\
& \quad - \alpha (\alpha - 4) \int_{\bbR^n} |x|^{-4} |x \cdot (\nabla f)(x)|^2 \, d^n x   \lb{0.1} \\
& \quad + \beta [(n - 4) (\alpha - 2) - \beta] 
\int_{\bbR^n} |x|^{-4} |f(x)|^2 \, d^n x, \quad f \in C^{\infty}_0(\bbR^n \backslash \{0\}).   \no 
\end{align}
A particular choice of values for $\alpha$ and $\beta$ in \eqref{0.1} yields known 
Hardy--Rellich-type inequalities, including the classical Rellich inequality and an inequality due to Schmincke. By locality, these inequalities extend to the situation where 
$\bbR^n$ is replaced by an arbitrary open set $\Omega \subseteq \bbR^n$ for functions 
$f \in C^{\infty}_0(\Omega \backslash \{0\})$. 

Perhaps more importantly, we will indicate that our method, in addition to being elementary, is quite flexible when it comes to a variety of generalized situations involving the inclusion of remainder terms and higher-order operators.  
\end{abstract}

\maketitle


\section{Introduction} \lb{s1}

{\it We dedicate this note with great pleasure to Helge Holden, whose wide range of 
contributions to a remarkable variety of areas in mathematical physics, stochastics, partial differential equations, and integrable systems, whose exemplary involvement with students, and whose tireless efforts on behalf of the mathematical community, deserve our utmost respect and admiration. Happy Birthday, Helge, we hope our modest contribution to Hardy--Rellich-type inequalities will give some joy.} 

\medskip

The celebrated (multi-dimensional) Hardy inequality,
\begin{equation}
\int_{\bbR^n} |(\nabla f)(x)|^2 \, d^n x \geq [(n-2)/2]^2 \int_{\bbR^n} |x|^{-2} |f(x)|^2 \, d^n x, 
\quad f \in C_0^{\infty}(\bbR^n \backslash \{0\}), \; n \in \bbN, \; n \geq 3,   \lb{1.1}
\end{equation}
and Rellich's inequality,
\begin{equation}
\int_{\bbR^n} |(\Delta f)(x)|^2 \, d^n x \geq [n(n-4)/4]^2 \int_{\bbR^n} |x|^{-4} |f(x)|^2 \, d^n x, 
\quad f \in C_0^{\infty}(\bbR^n \backslash \{0\}), \; n \in \bbN, \; n \geq 5,    \lb{1.2}
\end{equation}
the first two inequalities in an infinite sequence of higher-order Hardy-type inequalities, received enormous attention in the literature due to their ubiquity in self-adjointness and spectral theory problems associated with second and fourth-order differential operators with strongly singular coefficients, respectively (see, e.g., \cite{Al76}, \cite{AGG06}, \cite{BEL15}, \cite{Be89}, 
\cite[Sect.~1.5]{Da89}, \cite[Ch.~5]{Da95}, \cite{Ge84}, \cite{GP80}, \cite{GU98}, \cite{GS89}, \cite{Ka72}--\cite{KW72}, \cite[Ch.~II]{Re69}, \cite{Si83}). We refer to Remark \ref{r2.10} for a selection of Rellich inequality references and some pertinent monographs on Hardy's inequality. 

As one of our principal results we will derive the following two-parameter family of inequalities (a special case of inequality \eqref{1.5} below): 
If either $\alpha \leq 0$ or $\alpha \geq 4$, and $\beta \in \bbR$, then,
\begin{align}
\int_{\bbR^n} |(\Delta f)(x)|^2 \, d^n x 
& \geq [\alpha (n - \alpha) - 2 \beta] 
\int_{\bbR^n} |x|^{-2} |(\nabla f)(x)|^2 \, d^n x  \no \\
& \quad + \beta [(n - 4) (\alpha - 2) - \beta] 
\int_{\bbR^n} |x|^{-4} |f(x)|^2 \, d^n x,    \lb{1.3} \\
& \hspace*{2.3cm} f \in C^{\infty}_0(\bbR^n \backslash \{0\}), \; n \in \bbN, \; n \geq 2.   \no
\end{align} 

As will be shown, \eqref{1.3} contains Rellich's inequality \eqref{1.2}, and Schmincke's one-parameter family of inequalities, 
\begin{align} 
\int_{\bbR^n} |(\Delta f)(x)|^2 \, d^n x 
& \geq - s \int_{\bbR^n} |x|^{-2} |(\nabla f)(x)|^2 \, d^n x  \no \\
& \quad + [(n - 4)/4]^2 \big(4s + n^2\big) \int_{\bbR^n} |x|^{-4} |f(x)|^2 \, d^n x,   \lb{1.4} \\
& \hspace*{2.35cm} s \in \big[- 2^{-1} n(n - 4), \infty\big), \; n \geq 5,   \no 
\end{align} 
as special cases. By locality, the inequalities \eqref{1.1}--\eqref{1.4} naturally 
extend to the case where 
$\bbR^n$ is replaced by an arbitrary open set $\Omega \subset \bbR^n$ for functions 
$f \in C_0^{\infty}(\Omega \backslash \{0\})$ (without changing the constants in these inequalities). 

Our approach is based on factorizing even-order differential equations. More 
precisely, focusing on the 4th-order case for simplicity, we introduce the two-parameter 
$n$-dimensional homogeneous scalar differential expressions 
$T_{\alpha,\beta} := - \Delta + \alpha |x|^{-2} x \cdot \nabla + \beta |x|^{-2}$, 
$\alpha, \beta \in \bbR$, $x \in \bbR^n \backslash \{0\}$, $n \in \bbN$, 
$n \geq 2$, and its formal adjoint, denoted by $T_{\alpha,\beta}^+$. Nonnegativity of 
$T_{\alpha,\beta}^+ T_{\alpha,\beta}$ on $C_0^{\infty}(\bbR^n \backslash \{0\})$ then 
implies the fundamental inequality, 
\begin{align}
\int_{\bbR^n} [(\Delta f)(x)]^2 \, d^n x 
& \geq [(n - 4) \alpha - 2 \beta] \int_{\bbR^n} |x|^{-2} |(\nabla f)(x)|^2 \, d^n x  \no \\
& \quad - \alpha (\alpha - 4) \int_{\bbR^n} |x|^{-4} |x \cdot (\nabla f)(x)|^2 \, d^n x   \lb{1.5} \\
& \quad + \beta [(n - 4) (\alpha - 2) - \beta] 
\int_{\bbR^n} |x|^{-4} |f(x)|^2 \, d^n x, \quad f \in C^{\infty}_0(\bbR^n \backslash \{0\}),   \no 
\end{align}
which in turn contains inequality \eqref{1.3} as a special case. 

We conclude our note with a series of remarks putting our approach into proper context by indicating that our method is elementary and very flexible in handling a variety of generalized situations involving the inclusion of remainder terms and higher even-order differential operators.

\section{Factorizations and Hardy--Rellich-type Inequalities} \label{s2}

The principal inequality to be proven in this section is of the following form:

\begin{theorem} \lb{t2.1}
Let $\alpha, \beta \in \bbR$, and $f \in C^{\infty}_0(\bbR^n \backslash \{0\})$, 
$n \in \bbN$, $n \geq 2$. Then,
\begin{align}
\int_{\bbR^n} [(\Delta f)(x)]^2 \, d^n x 
& \geq [(n - 4) \alpha - 2 \beta] \int_{\bbR^n} |x|^{-2} |(\nabla f)(x)|^2 \, d^n x  \no \\
& \quad - \alpha (\alpha - 4) \int_{\bbR^n} |x|^{-4} |x \cdot (\nabla f)(x)|^2 \, d^n x  
\lb{2.8} \\
& \quad + \beta [(n - 4) (\alpha - 2) - \beta] 
\int_{\bbR^n} |x|^{-4} |f(x)|^2 \, d^n x.   \no 
\end{align}
In addition, if either $\alpha \leq 0$ or $\alpha \geq 4$, then,
\begin{align}
\int_{\bbR^n} |(\Delta f)(x)|^2 \, d^n x 
& \geq [\alpha (n - \alpha) - 2 \beta] 
\int_{\bbR^n} |x|^{-2} |(\nabla f)(x)|^2 \, d^n x  \no \\
& \quad + \beta [(n - 4) (\alpha - 2) - \beta] 
\int_{\bbR^n} |x|^{-4} |f(x)|^2 \, d^n x.   \lb{2.1}
\end{align} 
\end{theorem}
\begin{proof} 
Given $\alpha, \beta \in \bbR$ and $n \in \bbN$, $n \geq 2$, we introduce the two-parameter $n$-dimensional homogeneous scalar differential expressions  
\begin{equation}
T_{\alpha,\beta} := - \Delta + \alpha |x|^{-2} x \cdot \nabla + \beta |x|^{-2}, 
\quad x \in \bbR^n \backslash \{0\},   \lb{2.2}
\end{equation}
and its formal adjoint, denoted by $T_{\alpha,\beta}^+$,
\begin{equation}
T_{\alpha,\beta}^+ := - \Delta - \alpha |x|^{-2} x \cdot \nabla + [\beta - \alpha (n-2)] |x|^{-2}, 
\quad x \in \bbR^n \backslash \{0\}.     \lb{2.3}
\end{equation}
Assuming $f \in C^{\infty}_0(\bbR^n \backslash \{0\})$ throughout this proof, employing elementary multi-variable differential calculus, we proceed to the computation of $T_{\alpha,\beta}^+ T_{\alpha,\beta}$ (which, while entirely straightforward, may well produce some tears in the process),
\begin{align}
(T_{\alpha,\beta}^+ T_{\alpha,\beta} f)(x) &= (\Delta^2 f)(x) 
+ [(n - 4) \alpha - 2 \beta] |x|^{-2} (\Delta f)(x)    \no \\ 
& \quad + \alpha (4 - \alpha) |x|^{-4} \sum_{j,k = 1}^n x_j x_k f_{x_j, x_k }(x)  \no \\
& \quad + \big[- (n - 3) \alpha^2 + 2 (n - 2) \alpha + 4 \beta\big] |x|^{-4} x \cdot (\nabla f)(x)
\no \\
& \quad + \big[\beta^2 + 2 (n - 4) \beta - (n - 4) \alpha \beta\big] |x|^{-4} f(x).    \lb{2.4} 
\end{align}

Thus, choosing $f \in C^{\infty}_0(\bbR^n \backslash \{0\})$ real-valued from this point on and integrating by parts (observing the support properties of $f$, which results in vanishing surface terms) implies 
\begin{align}
0 & \leq \int_{\bbR^n} [(T_{\alpha,\beta} f)(x)]^2 \, d^n x 
=  \int_{\bbR^n} f(x) (T_{\alpha,\beta}^+ T_{\alpha,\beta} f)(x) \, d^n x   \no \\
& = \int_{\bbR^n} [(\Delta f)(x)]^2 \, d^n x 
+ [(n - 4) \alpha - 2 \beta] \int_{\bbR^n} \int_{\bbR^n} |x|^{-2} f(x) (\Delta f)(x) \, d^n x \no \\
& \quad + \alpha (\alpha - 4) \sum_{j,k = 1}^n \int_{\bbR^n} |x|^{-4} f(x) x_j x_k 
f_{x_j, x_k}(x) \, d^n x    \no \\ 
& \quad + \big[- (n -3) \alpha^2 + 2 (n - 2) \alpha + 4 \beta\big] 
\int_{\bbR^n} |x|^{-4} f(x) [x \cdot (\nabla f)(x)] \, d^n x    \no \\
& \quad + \big[\beta^2 + 2 (n - 4) \beta - (n - 4) \alpha \beta\big] 
\int_{\bbR^n} |x|^{-4} f(x)^2 \, d^n x.  \lb{2.5}
\end{align}

To simplify and exploit expression \eqref{2.5}, we make two observations. First, a standard integration by parts (again observing the support properties of $f$) yields
\begin{align}
\begin{split}
\int_{\bbR^n} |x|^{-2} f(x) (\Delta f)(x) \, d^n x  &= 
2 \int_{\bbR^n} |x|^{-4} f(x) (x \cdot (\nabla f)(x) \, d^n x    \lb{2.6} \\ 
& \quad - \int_{\bbR^n} |x|^{-2} |(\nabla f)(x)|^2 \, d^n x.    
\end{split}
\end{align}  
Similarly, one confirms that
\begin{align}
\begin{split}
\sum_{j,k = 1}^n \int_{\bbR^n} x_j x_k f(x) f_{x_j, x_k }(x) 
&= - (n - 3) \int_{\bbR^n} |x|^{-4} f(x) [x \cdot (\nabla f)(x)] \, d^n x \\
& \quad - \int_{\bbR^n} |x|^{-4} [x \cdot (\nabla f)(x)]^2 \, d^n x.    \lb{2.7} 
\end{split} 
\end{align}
Combining \eqref{2.5}--\eqref{2.7} then yields \eqref{2.8}. 

Since by Cauchy's inequality, 
\begin{equation}
- \int_{\bbR^n} |x|^{-4} [x \cdot (\nabla f)(x)]^2 \, d^n x 
\geq - \int_{\bbR^n} |x|^{-2} |(\nabla f)(x)|^2 \, d^n x,    \lb{2.9} 
\end{equation}
one concludes that as long as $\alpha (\alpha - 4) \geq 0$, that is, as long as either 
$\alpha \leq 0$ or $\alpha \geq 4$, one can further estimate \eqref{2.8} from below and thus arrive at inequality \eqref{2.1}. 
\end{proof}

As a special case of \eqref{2.1} one obtains Rellich's classical inequality in its original form as follows:
 
\begin{corollary} \lb{c2.2}
Let $n \in \bbN$, $n \geq 5$, and $f \in C^{\infty}_0(\bbR^n \backslash \{0\})$. Then,
\begin{align}
\int_{\bbR^n} |(\Delta f)(x)|^2 \, d^n x 
& \geq [n (n - 4)/4]^2 \int_{\bbR^n} |x|^{-4} |f(x)|^2 \, d^n x.   \lb{2.10}
\end{align} 
\end{corollary}
\begin{proof}
Choosing $\beta = \alpha (n - \alpha)/2$ in \eqref{2.1} results in
\begin{equation}
\int_{\bbR^n} |(\Delta f)(x)|^2 \, d^n x 
\geq G_n(\alpha) \int_{\bbR^n} |x|^{-4} |f(x)|^2 \, d^n x,   \lb{2.11}
\end{equation}
with
\begin{equation}
G_n(\alpha) = \alpha (n - \alpha) \{(n - 4) (\alpha -2 ) -  [\alpha (n - \alpha)/2]\}/2. 
\lb{2.12} 
\end{equation}
Maximizing $G_n(\alpha)$ with respect to $\alpha$ (it is advantageous to introduce the new variable $a = \alpha - 2$) yields maxima at 
\begin{equation}
\alpha_{\pm} = 2 \pm \big[\big(n^2/2\big) - 2n +4\big]^{1/2},   \lb{2.13} 
\end{equation}
and taking the constraints $\alpha \leq 0$ or $\alpha \geq 4$ into account results in 
$n \geq 5$. The fact 
\begin{equation}
G_n(\alpha_{\pm}) = [n (n - 4) /4]^2,     \lb{2.14} 
\end{equation}
then yields Rellich's inequality \eqref{2.10}.
\end{proof}

\medskip 

Inequality \eqref{2.8} also implies the following result: 

\begin{corollary} \lb{c2.3}
Let $n \in \bbN$ and $f \in C^{\infty}_0(\bbR^n \backslash \{0\})$. Then,
\begin{align}
\int_{\bbR^n} |(\Delta f)(x)|^2 \, d^n x 
& \geq \big(n^2/4\big) \int_{\bbR^n} |x|^{-2} |(\nabla f)(x)|^2 \, d^n x, \quad n \geq 8,     \lb{2.15}
\end{align} 
and 
\begin{align}
\int_{\bbR^n} |(\Delta f)(x)|^2 \, d^n x 
& \geq 4(n-4) \int_{\bbR^n} |x|^{-2} |(\nabla f)(x)|^2 \, d^n x, \quad 5 \leq n \leq 7.    \lb{2.15A}
\end{align} 
In addition,
\begin{align}
\int_{\bbR^n} |(\Delta f)(x)|^2 \, d^n x 
& \geq \big(n^2/4\big) \int_{\bbR^n} |x|^{-4} |x \cdot (\nabla f)(x)|^2 \, d^n x, \quad 
n \geq 2.    \lb{2.15a}
\end{align} 
\end{corollary}
\begin{proof}
Again, we chose $f \in C^{\infty}_0(\bbR^n \backslash \{0\})$ real-valued for simplicity throughout this proof. The choice $\beta = (n - 4)(\alpha - 2)$  in \eqref{2.8} then results in
\begin{align}
\begin{split}
\int_{\bbR^n} |(\Delta f)(x)|^2 \, d^n x 
& \geq (n - 4)(4 - \alpha) \int_{\bbR^n} |x|^{-2} |(\nabla f)(x)|^2 \, d^n x  \\
& \quad - \alpha (\alpha - 4) \int_{\bbR^n} |x|^{-4} |x \cdot (\nabla f)(x)|^2 \, d^n x.   
\lb{2.16}
\end{split}
\end{align}
If in addition $\alpha < 0$, then applying Cauchy's inequality to the 2nd term on the 
right-hand side of \eqref{2.16} yields
\begin{equation}
\int_{\bbR^n} |(\Delta f)(x)|^2 \, d^n x 
\geq H_n(\alpha) \int_{\bbR^n} |x|^{-2} |(\nabla f)(x)|^2 \, d^n x,    \lb{2.16a}
\end{equation} 
where $H_n(\alpha) = (n - 4 + \alpha)(4 - \alpha)$. Maximizing $H_n$ with respect to 
$\alpha$ subject to the constraint $\alpha < 0$ yields a maximum at 
$\alpha_1 = (8 - n)/2$, with $H_n((8 - n)/2) = n^2/4$, 
implying inequality \eqref{2.15} for $n \geq 9$. On the other hand, choosing 
$\alpha = 0$ in \eqref{2.16} yields 
\begin{equation} 
\int_{\bbR^n} |(\Delta f)(x)|^2 \, d^n x 
\geq 4 (n - 4) \int_{\bbR^n} |x|^{-2} |(\nabla f)(x)|^2 \, d^n x.   \lb{2.19} 
\end{equation} 
Since $4(n-4) = n^2/4$ for $n=8$, this proves \eqref{2.15}. Actually, one can 
arrive at \eqref{2.15} much quicker, but since we will subsequently use \eqref{2.16}, we kept the above argument in this proof: Indeed, choosing $\beta = 0$ in 
\eqref{2.1} yields
\begin{equation}
\int_{\bbR^n} |(\Delta f)(x)|^2 \, d^n x 
\geq \alpha (n - \alpha) \int_{\bbR^n} |x|^{-2} |(\nabla f)(x)|^2 \, d^n x.    \lb{2.19B}
\end{equation} 
Maximizing $F_n(\alpha) = \alpha (n - \alpha)$ with respect to $\alpha$ yields a 
maximum at $\alpha_1 = n/2$, and subjecting it to the constraint $\alpha \geq 4$ proves \eqref{2.15}. 

Choosing $\alpha = 4$, $\beta = 0$ in \eqref{2.8} yields \eqref{2.15A}. 

For $n \geq 2$ and $(4 - n) < \alpha < 4$, applying Cauchy's inequality to the 1st term on the right-hand side of \eqref{2.16} now yields
\begin{equation}
\int_{\bbR^n} |(\Delta f)(x)|^2 \, d^n x 
\geq K_n(\alpha) \int_{\bbR^n} |x|^{-4} |x \cdot (\nabla f)(x)|^2 \, d^n x,    
\lb{2.19a}
\end{equation}
where $K_n(\alpha) = - (\alpha + n - 4)(\alpha - 4)$. Maximizing $K_n$ subject to the constraint $(4 - n) < \alpha < 4$ yields a maximum at $\alpha_1 = (8 - n)/2$, with 
$K_n((8 - n)/2) = n^2/4$, implying \eqref{2.15a}.
\end{proof}

We conclude with a series of remarks that put our approach into proper context and point out natural continuations into various other directions. 

\begin{remark} \lb{r2.3a}
$(i)$ The constant in inequality \eqref{2.10} is known to be optimal, see, for instance,
\cite[p.~222]{BEL15}, \cite{DH98}, \cite{Mi00}, \cite{Ok96}, \cite{TZ07}, \cite{Ya99}. \\
$(ii)$ A sequence of extensions of \eqref{2.15}, valid for $n \geq 5$, and for bounded 
domains containing $0$, was derived by 
Tertikas and Zographopoulos \cite[Theorem~1.7]{TZ07}. Moreover, an extension 
of inequality \eqref{2.15} valid for $n = 4$ and for bounded open domains containing $0$ was proved by \cite[Theorem~2.1\,(b)]{AGS06}. An alternative inequality whose special cases also imply Rellich's inequality \eqref{2.10} and inequality \eqref{2.15} appeared in \cite{Co09}. Thus, while the constant $n^2/4$ in \eqref{2.15} is known to be optimal (cf.\ \cite{TZ07} for $n \geq 5$), the constant $4(n-4)$ in \eqref{2.15A} is not, the sharp constant being known to be $n^2/4$ (also for $n=4$, cf.\ \cite{AGS06}). 
\hfill $\diamond$ 
\end{remark}

Next, we comment on a special case of inequality \eqref{2.1} originally due to 
Schmincke \cite{Sc72}:

\begin{remark} \lb{r2.4}
The choice $\beta = 2^{-1}(n - 4)[\alpha -2 - 4^{-1}(n - 4)]$, and the introduction of the new variable
\begin{equation}
s = s(\alpha) = \alpha^2 - 4 \alpha - 2^{-1} n (n - 4),    \lb{2.19A} 
\end{equation}
renders the two-parameter inequality \eqref{2.1} into Schmincke's one-parameter inequality
\begin{align} 
\int_{\bbR^n} |(\Delta f)(x)|^2 \, d^n x 
& \geq - s \int_{\bbR^n} |x|^{-2} |(\nabla f)(x)|^2 \, d^n x  \no \\
& \quad + [(n - 4)/4]^2 \big(4s + n^2\big) \int_{\bbR^n} |x|^{-4} |f(x)|^2 \, d^n x,   \lb{2.19b} \\
& \hspace*{2.4cm} s \in \big[- 2^{-1} n(n - 4), \infty\big), \; n \geq 5.    \no 
\end{align} 
Here the requirements $\alpha \leq 0$, equivalently, $\alpha \geq 4$, both yield the range requirement for $s$ in the form $s \in \big[- 2^{-1} n(n - 4), \infty\big)$. Inequality \eqref{2.19b} 
is precisely the content of Lemma~2 in Schmincke \cite{Sc72}, in particular, \eqref{2.1} thus recovers Schmincke's result. Moreover, assuming $n \geq 5$ (the case $n = 4$ being trivial) permits the value $s=0$ and hence implies Rellich's inequality \eqref{2.10}. If $n \geq 8$, the value 
$s= - n^2/4$ is permitted, yielding inequality \eqref{2.15}. Finally, for 
$5 \leq n \leq 7$, $s \in \big[- 2^{-1} n(n - 4), \infty\big)$ and $4s + n^2 \geq 0$  permit one to choose $s = - n(n - 4)/2$ and hence to conclude
\begin{align}
\int_{\bbR^n} |(\Delta f)(x)|^2 \, d^n x & \geq 2^{-1}n(n - 4) 
\int_{\bbR^n} |x|^{-2} |(\nabla f)(x)|^2 \, d^n x   \no \\
& \quad + [(n - 4)/4]^2 \big(4s + n^2\big) \int_{\bbR^n} |x|^{-4} |f(x)|^2 \, d^n x   
\eqref{2.19B} \\
& \geq 2^{-1}n(n - 4) \int_{\bbR^n} |x|^{-2} |(\nabla f)(x)|^2 \, d^n x, \quad 
5 \leq n \leq 7,    \lb{2.19c} 
\end{align}
but inequality \eqref{2.15A} is strictly superior to \eqref{2.19c}. Hence the two-parameter version \eqref{2.1} yields the better result \eqref{2.15A}, even though, the latter is not optimal either as mentioned in the previous Remark \ref{r2.3a}\,$(ii)$. 
\hfill $\diamond$ 
\end{remark}

\begin{remark} \lb{r2.5}
Since all differential expressions employed are local, and only integration by parts was involved in deriving \eqref{2.5} (cf., e.g., \cite[Remark~4]{Ka80} in this context), the estimates \eqref{2.8}, \eqref{2.1}, \eqref{2.10}, \eqref{2.15}--\eqref{2.15a}, \eqref{2.19b}, all extend to the case where $\bbR^n$ is replaced by an arbitrary open set $\Omega \subset \bbR^n$ for functions 
$f \in C_0^{\infty}(\Omega \backslash \{0\})$ (without changing the constants in 
these inequalities).  \hfill $\diamond$
\end{remark}

\begin{remark} \lb{r2.6}
Since $C_0^{\infty}(\bbR^n \backslash \{0\})$ is dense in $H^2(\bbR^n)$ if and only if 
$n \geq 4$, equivalently, $- \Delta\big|_{C_0^{\infty}(\bbR^n \backslash \{0\})}$ is essentially 
self-adjoint in $L^2(\bbR^n)$ if and only if $n\geq 4$ (see, e.g., \cite[p.~412--413]{EE89}, 
see also \cite{KSWW75}, \cite{Sc72}), 
Rellich's inequality \eqref{2.10} extends from $C_0^{\infty}(\bbR^n \backslash \{0\})$ to $H^2(\bbR^n)$ for $n \geq 5$, and inequalities \eqref{2.15}, and \eqref{2.15a} extend from $C_0^{\infty}(\bbR^n \backslash \{0\})$ to $H^2(\bbR^n)$ for $n \geq 4$. 
\hfill $\diamond$
\end{remark}

\begin{remark} \lb{r2.7}
This factorization approach was originally employed in the context of the classical Hardy inequality in \cite{GP80} (and some of its logarithmic refinements in \cite{Ge84}). Without repeating the analogous steps in detail we just mention that given $n \in \bbN$, $n \geq 3$, 
$\alpha \in \bbR$, one introduces the one-parameter family of homogeneous vector-valued differential expressions
\begin{equation}
T_{\alpha} := \nabla + \alpha |x|^{-2} x, \quad  x \in \bbR^n \backslash \{0\},  \lb{2.20}
\end{equation}
with formal adjoint, denoted by $T_{\alpha}^+$, 
\begin{equation}
T_{\alpha}^+ = - {\rm div(\, \cdot \,)} + \alpha |x|^{-2} x \, \cdot, \quad  
\; x \in \bbR^n \backslash \{0\},    \lb{2.21}
\end{equation}
such that (e.g., on $C_0^{\infty}(\bbR^n \backslash \{0\})$-functions), 
\begin{equation}
T_{\alpha}^+ T_{\alpha} = - \Delta + \alpha (\alpha + 2 - n) |x|^{-2}. 
\lb{2.22} 
\end{equation}
Thus, for $f \in C_0^{\infty}(\bbR^n \backslash \{0\})$, 
\begin{align}
\begin{split}
0 & \leq \int_{\bbR^n} |T_{\alpha} f)(x)|^2 \, d^n x 
= \int_{\bbR^n} \ol{f(x)} (T_{\alpha}^+ T_{\alpha} f)(x) \, d^n x     \\
&= \int_{\bbR^n} |(\nabla f)(x)|^2 \, d^n x + \alpha (\alpha + 2 - n)
\int_{\bbR^n} |x|^{-2} |f(x)|^2 \, d^n x,   \lb{2.23} 
\end{split} 
\end{align}
and hence,
\begin{equation}
\int_{\bbR^n} |(\nabla f)(x)|^2 \, d^n x \geq \alpha [(n - 2) - \alpha]
\int_{\bbR^n} |x|^{-2} |f(x)|^2 \, d^n x.    \lb{2.24} 
\end{equation}
Maximizing $\alpha [(n - 2) - \alpha]$ with respect to $\alpha$ yields the classical  Hardy inequality,
\begin{equation}
\int_{\bbR^n} |(\nabla f)(x)|^2 \, d^n x \geq [(n - 2)/2]^2 
\int_{\bbR^n} |x|^{-2} |f(x)|^2 \, d^n x, \quad f \in C_0^{\infty}(\bbR^n \backslash \{0\}), \; n \geq 3.  
\lb{2.25} 
\end{equation}
Again, it is well-known that the constant in \eqref{2.25} is optimal (cf., e.g., \cite{Ya99}). 
\hfill $\diamond$
\end{remark}

Actually, our factorization approach also yields a known improvement of Hardy's inequality (see, e.g., \cite[Theorem~1.2.5]{BEL15}, specializing it to $p=2$, $\varepsilon =0$). Next, we briefly sketch the corresponding argument.

\begin{remark} \lb{r2.8}
Given $n \in \bbN$, $n \geq 3$, $\alpha \in \bbR$, one introduces the following modified one-parameter family of homogeneous vector-valued differential expressions
\begin{equation}
\wti T_{\alpha} := \big(|x|^{-1} x\big) \cdot \nabla + \alpha |x|^{-1}, 
\quad  x \in \bbR^n \backslash \{0\},  \lb{2.25a}
\end{equation}
with formal adjoint, denoted by $\big(\wti T_{\alpha}\big)^+$, 
\begin{equation}
\big(\wti T_{\alpha}\big)^+= - \big(|x|^{-1} x\big) \cdot \nabla  
+ (\alpha - n +1) |x|^{-1}, \quad  \; x \in \bbR^n \backslash \{0\}.    \lb{2.26a}
\end{equation}
Exploiting the identities (for $f \in C_0^{\infty}(\bbR^n \backslash \{0\})$, for simplicity), 
\begin{align}
& \big[|x|^{-1} x \cdot \nabla\big] \big[|x|^{-1} x \cdot (\nabla f)(x)\big] = 
|x|^{-2} \sum_{j,k=1}^n x_j x_k f_{x_j, x_k}(x), \quad  x \in \bbR^n \backslash \{0\},  \lb{2.27a} \\
& x \cdot \nabla \big(|x|^{-1} f(x)\big) = |x|^{-1} [x \cdot (\nabla f)(x)] - |x|^{-1} f(x), 
\quad  x \in \bbR^n \backslash \{0\},  \lb{2.28a} 
\end{align}
one computes (e.g., on $C_0^{\infty}(\bbR^n \backslash \{0\})$-functions), 
\begin{align}
& \big(\wti T_{\alpha}\big)^+ \wti T_{\alpha} 
= - |x|^{-2} \sum_{j,k=1}^n x_j x_k \partial_{x_j} \partial_{x_k} 
- (n-1) |x|^{-2} [x \cdot (\nabla f)(x)] + \alpha (\alpha + 2 - n) |x|^{-2},   \no \\ 
& \hspace*{10.5cm} x \in \bbR^n \backslash \{0\}.    \lb{2.29a} 
\end{align}
Thus, appropriate integration by parts yield 
\begin{align}
0 & \leq \int_{\bbR^n} \big|\big(\wti T_{\alpha} f\big)(x)\big|^2 \, d^n x 
= \int_{\bbR^n} \ol{f(x)} \big(\big(\wti T_{\alpha}\big)^+ \wti T_{\alpha} f\big)(x) \, d^n x   \no \\
&= - \int_{\bbR^n} |x|^{-2} \bigg\{\sum_{j,k=1}^n x_j x_k \ol{f(x)} f_{x_j, x_k}(x) 
+ (n-1) \ol{f(x)} [x \cdot (\nabla f)(x)]    \no \\
& \hspace*{2.4cm} - \alpha (\alpha - n + 2) |f(x)|^2\bigg\} \, d^n x,   \no \\
&= \int_{\bbR^n} |x|^{-2} \bigg\{|[x \cdot (\nabla f)(x)]|^2 
+ \alpha (\alpha - n + 2)|f(x)|^2\bigg\}, \quad f \in C_0^{\infty}(\bbR^n \backslash \{0\}).     \lb{2.30a} 
\end{align}
Here we used 
\begin{align}
\begin{split}
& \sum_{j,k=1}^n \int_{\bbR^n} |x|^{-2} x_j x_k \ol{f(x)} f_{x_j, x_k}(x) \, d^n x 
= - \int_{\bbR^n} |x|^{-2} |[x \cdot (\nabla f)(x)]|^2 \, d^n x   \lb{2.31a} \\
& \quad - (n-1) \int_{\bbR^n} |x|^{-2} \ol{f(x)} [x \cdot (\nabla f)(x)] \, d^n x, 
\quad f \in C_0^{\infty}(\bbR^n \backslash \{0\}).   
\end{split}
\end{align}
Thus,
\begin{equation}
\int_{\bbR^n} \big|\big[|x|^{-1} x \cdot \nabla f](x)\big|^2 \, d^n x \geq \alpha [(n - 2) - \alpha]
\int_{\bbR^n} |x|^{-2} |f(x)|^2 \, d^n x.    \lb{2.32a} 
\end{equation}
Maximizing $\alpha [(n - 2) - \alpha]$ with respect to $\alpha$ yields the improved Hardy inequality,
\begin{equation}
\int_{\bbR^n} \big|\big[|x|^{-1} x \cdot \nabla f](x)\big|^2 \, d^n x \geq [(n - 2)/2]^2 
\int_{\bbR^n} |x|^{-2} |f(x)|^2 \, d^n x, \quad f \in C_0^{\infty}(\bbR^n \backslash \{0\}), \; n \geq 3.     \lb{2.33a} 
\end{equation}
(By Cauchy's inequality, \eqref{2.33a} implies the classical Hardy inequality \eqref{2.25}.)
Again, it is known that the constant in \eqref{2.33a} is optimal (cf., e.g., \cite[Theorem~1.2.5]{BEL15}). 
\hfill $\diamond$
\end{remark}

\begin{remark} \lb{r2.9}
The case of Rellich (and Hardy) inequalities in the half-line case is completely analogous (and much more straightforward): Consider the differential expressions
\begin{align}
T = - \f{d^2}{dx^2} + \f{\alpha}{x} \f{d}{dx} + \f{\beta}{x^2}, \quad 
T^+ = - \f{d^2}{dx^2} - \f{\alpha}{x} \f{d}{dx} + \f{\alpha + \beta}{x^2},    \lb{2.27} 
\end{align}
with $\alpha, \beta \in \bbR$, which are formal adjoints to each other. One verifies,
\begin{equation}
T^+ T = \f{d^4}{dx^4} + \f{\alpha - \alpha^2 - 2 \beta}{x^2} \f{d^2}{dx^2} 
+ \f{2 \alpha^2 - 2 \alpha + 4 \beta}{x^3} \f{d}{dx} + \f{3 \alpha \beta + \beta^2 - 6 \beta}{x^4}, 
\lb{2.28} 
\end{equation}
and hence upon some integrations by parts,
\begin{align} 
0 &\leq \int_0^{\infty} (T f)(x)^2 \, dx = \int_0^{\infty} f(x) (T^+ T f)(x) \, dx    \no \\ 
& = \int_0^{\infty} [f''(x)]^2 \, dx 
- \big(\alpha - \alpha^2 - 2 \beta\big) \int_0^{\infty} \f{[f'(x)]^2}{x^2} \, dx   \no \\
& \quad + \beta (3 \alpha + \beta - 6) \int_0^{\infty} \f{f(x)^2}{x^4} \, dx, \quad 
f \in C_0^{\infty} ((0,\infty)),    \lb{2.29} 
\end{align}
choosing $f$ real-valued (for simplicity and w.l.o.g.). Thus, one obtains,
\begin{align}
\int_0^{\infty} |f''(x)|^2 \, dx 
&\geq \big(\alpha - \alpha^2 - 2 \beta\big) \int_0^{\infty} \f{|f'(x)|^2}{x^2} \, dx  \no \\
& \quad + \beta (6 - \beta - 3 \alpha) \int_0^{\infty} \f{|f(x)|^2}{x^4} \, dx,     \lb{2.30} \\
& \hspace*{1.3cm} f \in C_0^{\infty} ((0,\infty)), \; \alpha, \beta \in \bbR .    \no 
\end{align}
Choosing $\beta = \big(\alpha - \alpha^2\big)/2$ yields the Rellich-type inequality 
\begin{align}
& \int_0^{\infty} |f''(x)|^2 \, dx \geq \big[3 \alpha - (19/4) \alpha^2 + 2 \alpha^3 - (1/4) \alpha^4\big] 
\int_0^{\infty} \f{|f(x)|^2}{x^4} \, dx,   \no \\
& \hspace*{7.9cm} f \in C_0^{\infty} ((0,\infty)).&    \lb{2.31} 
\end{align}

Introducing $F(\alpha) = 3 \alpha - (19/4) \alpha^2 + 2 \alpha^3 - (1/4) \alpha^4    
= - (1/4) (\alpha - 4) (\alpha - 3) (\alpha - 1) \alpha$, $\alpha \in \bbR$, one verifies that 
$F(2 + \gamma) = F(2 - \gamma)$, $\gamma \in \bbR$, and factors its derivative as 
\begin{align}
\begin{split} 
F'(\alpha) &= 3 - (19/2) \alpha + 6 \alpha^2 - \alpha^3    \\
&= - (\alpha - 2) \big(\alpha - 2 + (5/2)^{1/2}\big) \big(\alpha - 2 - (5/2)^{1/2}\big).   \lb{2.32}
\end{split} 
\end{align}
One notes that $\alpha_1 = 2$ yields a local minimum with $F(2) = - 1$, 
$\alpha_2 = 2 - (5/2)^{1/2}$ and $\alpha_3 = 2 + (5/2)^{1/2}$ both yield local maxima of equal value, that is, $F(\alpha_2) = F(\alpha_3) = \f{9}{16}$. 
Thus, one obtains Rellich's inequality for the half-line in the form 
\begin{align}
\int_0^{\infty} |f''(x)|^2 \, dx \geq \f{9}{16} 
\int_0^{\infty} \f{|f(x)|^2}{x^4} \, dx, \quad f \in C_0^{\infty} ((0,\infty)).    \lb{2.33} 
\end{align}
We refer to Birman \cite[p.~46]{Bi66} (see also Glazman \cite[p.~83--84]{Gl65}), who presents a sequence of higher-order Hardy-type inequalities on $(0,\infty)$ whose second member coincides with \eqref{2.33}. For a variant of 
\eqref{2.33} on the interval $(0,1)$ we refer to \cite[p.~114]{Da95}; the case of 
higher-order Hardy-type inequalities for general interval is also considered in \cite{Ow99}. We will reconsider this sequence of higher-order Hardy-type inequalities in \cite{GLMW18}.

In addition, choosing $\beta = 0$ or $\beta = 6 - 3 \alpha$ and subsequently maximizing 
with respect to $\alpha$ yields in either case 
\begin{equation}
\int_0^{\infty} |f''(x)|^2 \, dx \geq \f{1}{4} \int_0^{\infty} \f{|f'(x)|^2}{x^2} \, dx, \quad 
f \in C_0^{\infty} ((0,\infty)),     \lb{2.34} 
\end{equation}
however, this is just Hardy's inequality \cite{Ha19}, \cite{Ha20}, with $f$ replaced by $f'$. \hfill $\diamond$
\end{remark}

\begin{remark} \lb{r2.10}
While we basically focused on inequalities in $L^2(\bbR^n)$ (see, however, Remark \ref{r2.5}), much of the recent work on Rellich and higher-order Hardy inequalities aims at $L^p(\Omega)$ for open sets $\Omega \subset \bbR^n$ (frequently, $\Omega$ is bounded with $0 \in \Omega$), $p \in [1,\infty)$, appropriate remainder terms (the latter often associated with logarithmic refinements or with boundary terms), higher-order Hardy--Rellich inequalities, and the inclusion of magnetic fields. The enormous number of references on this subject, especially, in the context of Hardy-type inequalities, makes it impossible to achieve any reasonable level of completeness in such a short note as the underlying one. Hence we felt we had to restrict ourselves basically to Rellich and higher-order Hardy inequality references only and thus we refer, for instance, to \cite{AGS06}, \cite{Al76}, \cite{Av16}, \cite{Av16a}, \cite[Ch.~6]{BEL15}, \cite{Ba06}, \cite{Ba07}, \cite{BT06}, \cite{Be09}, \cite{BCG10}, \cite{DH98}, \cite{DHA04}, \cite{DHA04a}, \cite{DHA12}, \cite{EE16}, \cite{Ev09}, \cite{EL05}, \cite{EL07}, \cite{Ga06}, \cite{GGM03}, \cite{GM11}, \cite{Ka84}, \cite{MSS15}, 
\cite{MSS16}, \cite{Mo12}, \cite{Mu14}, \cite{Ow99}, \cite{Pa91}, \cite{RS16a}, \cite{TZ07}, \cite{Xi14}, \cite{XY09}, and the extensive literature cited therein. For the case of Hardy-type inequalities we only refer to the standard monographs such as, \cite{BEL15}, \cite{KMP07}, \cite{KP03}, and \cite{OK90}.

In this context we emphasize once again that the factorization method is entirely independent of the choice of domain $\Omega$. Indeed, factorizations in the context of Hardy's inequality in balls with optimal constants and logarithmic correction terms were already studied in \cite{Ge84}, \cite{GP80}, based on prior work in \cite{Ka72}, \cite{KSWW75}, and \cite{KW72}, although this appears to have gone unnoticed in the recent literature on this subject. For instance, one can introduce iterated logarithms of the form for $\gamma > 0$, $x\in \bbR^n$, $|x| < \gamma$, 
\begin{align}
& (- \ln(|x|/\gamma))_0 =1, \no \\
& (- \ln(|x|/\gamma))_1 = (- \ln(|x|/\gamma)), \no \\
& (- \ln(|x|/\gamma))_{k+1} = \ln((- \ln(|x|/\gamma))_k), \quad k \in \bbN,   \lb{2.39} 
\end{align}
and replace $\nabla$ by $T_{\alpha_m, y}$, where 
\begin{align}
& T_{\alpha_m, y}  = \nabla  + 2^{-1} |x-y|^{-2}\bigg\{(n-2) + \sum_{j=1}^m \prod_{k=1}^j 
[(- \ln(|x-y|/\gamma))_k]^{-1}    \no \\
& \hspace*{4.2cm} - \alpha_m \prod_{k=1}^m [(-\ln(|x-y|/\gamma))_k]^{-1}\bigg\} (x-y),  
\lb{2.40} \\
& \hspace*{5cm} 0 < |x| < r, \; r < \gamma, \; \alpha_m \geq 0, \; m \in\bbN,  \no \\
& T_{\alpha_0, y} = \nabla + 2^{-1} (n-2-\alpha_0) |x-y|^{-2} (x-y), \quad 
0 < |x| < r, \; \alpha_0 \geq 0, \; m=0.   \lb{2.41} 
\end{align}
Then with $T_{\alpha_m, y}^+$ the formal adjoint of $T_{\alpha_m, y}$, one obtains for 
$f \in C_0^{\infty}(B_n(y; r) \backslash \{y\})$
\begin{align}
& (T_{\alpha_m,y}^+ T_{\alpha_m,y} f)(x) = (- \Delta f)(x) - 4^{-1} |x-y|^{-2} \bigg\{(n-2)^2  \lb{2.42} \\
& \quad + \sum_{j=1}^m \prod_{k=1}^j [(- \ln(|x-y|/\gamma))_k]^{-2} f(x) 
- \alpha_m^2 \prod_{k=1}^m [(- \ln(|x-y|/\gamma))_k]^{-2} \bigg\}f(x), \quad m \in \bbN,  \no \\
& (T_{\alpha_0,y}^+ T_{\alpha_0,y} f)(x) = (- \Delta f)(x) - 4^{-1} \big[(n-2)^2 - \alpha_0^2\big] |x-y|^{-2} f(x), \quad m=0.   \lb{2.43} 
\end{align} 
(Here $B_n(x_0;r_0)$ denotes the open ball in $\bbR^n$ with center $x_0 \in \bbR^n$ and radius $r_0>0$.) In particular, letting $r_0 \downarrow 0$ and $r_1 \uparrow r$ in 
\cite[Lemma~1]{Ge84} implies 
\begin{align}
& 0 \leq \int_{B(y;r)} |(T_{\alpha_m, y} f)(x)|^2 = 
\int_{B(y;r)} \bigg\{|(\nabla f)(x)|^2 - 4^{-1} |x-y|^{-2} \bigg[(n-2)^2   \no \\
& \quad
+ \sum_{j=1}^m \prod_{k=1}^j [(- \ln(|x-y|/\gamma))_k]^{-2} 
- \alpha_m^2 \prod_{k=1}^m [(- \ln(|x-y|/\gamma))_k]^{-2} \bigg] |f(x)|^2 \bigg\}  \, d^n x ,   
\lb{2.44} \\
& \hspace*{4.9cm} 0 < r < \gamma, \; f \in C_0^{\infty}(B(y;r) \backslash \{y\}), \; 
m \in \bbN \cup \{0\},    \no 
\end{align} 
and hence (with $\alpha_m = 0$), 
\begin{align}
& \int_{B(y;r)} |(\nabla f)(x)|^2 \, d^n x \geq 4^{-1} \int_{B(y;r)} |x-y|^{-2} \bigg\{(n-2)^2   \no \\
& \hspace*{5.5cm} 
+ \sum_{j=1}^m \prod_{k=1}^j [(- \ln(|x-y|/\gamma))_k]^{-2}\bigg\} |f(x)|^2 \, d^n x,   \lb{2.45} \\
& \hspace*{5.1cm} 0 < r < \gamma, \; f \in C_0^{\infty}(B(y;r) \backslash \{y\}), \; 
m \in \bbN \cup \{0\}.   \no 
\end{align}
(Following standard practice, a product, resp., sum over an empty index set is defined to equal $1$, resp., $0$.)
In analogy to Remark \ref{r2.5}, inequality \eqref{2.45} extends to arbitrary open bounded sets 
$\Omega \subset \bbR^n$ as long as $\gamma$ is chosen sufficiently large (e.g., larger than 
the diameter of $\Omega$).  The constants in \eqref{2.44} are best possible as it is well-known that the operators \eqref{2.42}, \eqref{2.43} are nonnegative if and only 
if $\alpha_m^2 \geq 0$, $m \in \bbN \cup \{0\}$ (cf., e.g., \cite[p.~99]{Ge84} or \cite[Theorem~2.2]{GU98}). (They are unbounded from below for $\alpha_m^2 < 0$, $m \in \bbN \cup \{0\}$, permitting temporarily a continuation to negative values of $\alpha^2$ on the right-hand sides of \eqref{2.42} and \eqref{2.43}.)

For the special half-line case we also refer to \cite{GU98}.

Higher-order logarithmic refinements of the multi-dimensional 
Hardy--Rellich-type inequality appeared in \cite[Theorem~2.1]{AGS06}, and a 
sequence of such multi-dimensional Hardy--Rellich-type inequalities, 
with additional generalizations, appeared in \cite[Theorems~1.8--1.10]{TZ07}. 
\hfill $\diamond$
\end{remark}

We conclude this section by mentioning that factorization also works for other singular interactions, for instance, for point dipole interactions, where $|x|^{-2}$ is replaced by 
$|x|^{-3}(d \cdot x)$, with $d \in \bbR^n$ a constant vector. Moreover, it applies to higher-order Hardy-type inequalities where $- \Delta$ is replaced by $(- \Delta)^\ell$, $\ell \in \bbN$. We defer all this to future investigations \cite{GLM18}.

\section{An Application of Rellich's Inequality} \label{s3}

In our final section we sketch an application to lower semiboundedness and to form boundedness for interactions with countably many strong singularities. 

To keep matters short we will just aim at the particular case $(-\Delta)^2 + W$, 
where $W$ has countably many strong singularities. We start by recalling an abstract version of a result of Morgan \cite{Mo79} as described in 
\cite{GMNT16}: 

\begin{theorem} \lb{t3.1}
 Suppose that $T$, $W$ are self-adjoint operators in $\cH$ such that  
$\dom\big(|T|^{1/2}\big)\subseteq\dom\big(|W|^{1/2}\big)$,
and let $c, d \in (0,\infty)$, $e \in [0,\infty)$.
Moreover, suppose $\Phi_j \in \cB(\cH)$, $j \in J$, $J \in \bbN$ an index set,  leave $\dom\big(|T|^{1/2}\big)$ invariant, that is, $\Phi_j \dom\big(|T|^{1/2}\big) \subseteq \dom\big(|T|^{1/2}\big)$, $j \in J$, 
and satisfy the following conditions $(i)$--$(iii)$: \\[1mm]  
$(i)$ $\sum_{j  \in J} \Phi_j^* \Phi_j \leq I_{\cH}$. \\[1mm] 
$(ii)$ $\sum_{j  \in J}  \Phi_j^* |W| \Phi_j \ge c^{-1} |W|$ on $\dom\big(|T|^{1/2}\big)$.\\[1mm] 
$(iii)$ $\sum_{j  \in J}  \| |T|^{1/2} \Phi_j f\|_{\cH}^2 \le 
d \| |T|^{1/2} f\|_{\cH}^2 
+ e\|f\|_{\cH}^2$, \, $f \in \dom\big(|T|^{1/2}\big)$. \\[1mm] 
Then, 
\begin{equation}
\big\| |W|^{1/2} \Phi_j f \big\|_{\cH}^2 \le a \big\||T|^{1/2} \Phi_j f \big\|_{\cH}^2 
+ b \|\Phi_j f\|_{\cH}^2, \quad f \in\dom(|T|^{1/2}), \; j\in J, 
\end{equation} 
implies 
\begin{equation} 
\big\| |W|^{1/2} f \big\|_{\cH}^2 \le a\,c\,d \big\||T|^{1/2} f \big\|_{\cH}^2 
+ [a\,c\,e + b\, c] \|f\|_{\cH}^2, \quad f \in\dom(|T|^{1/2}). 
\end{equation} 
\end{theorem}

Thus, the key for applications would be to have $c$ and $d$ arbitrarily close to $1$ such that if $a < 1$, also $acd < 1$. 

If $W$ is local and $\Phi_j$ represents the operator of multiplication with 
``bump functions'' $\phi_j$, $j \in J\subseteq \bbN$, such that $\phi_j$, 
$j \in J$ is a family of smooth, real-valued functions defined 
on $\bbR^n$ satisfying that for each $x \in \bbR^n$, there exists an open neighborhood $U_x \subset \bbR^n$ of $x$ such that there exist only finitely many indices $k \in J$ with $\supp \, (\phi_k) \cap U_x \neq \emptyset$ and  $\phi_k|_{U_x} \neq 0$, as well as 
\begin{equation}
\sum_{j\in J} \phi_j(x)^2=1, \quad x \in \bbR^n
\end{equation}    
(the sum over $j \in J$ being finite). Then $\Phi_j$ and $W$ commute and hence 
\begin{equation} 
\sum_{j  \in J} \Phi_j^* \Phi_j = I_{\cH} \, \text{ and } \, 
\sum_{j  \in J}  \Phi_j^* |W| \Phi_j = |W| \, \text{ on } \, 
\dom\big(|T|^{1/2}\big) 
\end{equation} 
yield condition $(i)$  and also $(ii)$ with $c = 1$ of Theorem \ref{t3.1}. 

Next, we will illustrate a typical situation where for all $\varepsilon > 0$, one can actually choose $d = 1 + \varepsilon$. 

Consider $T = (- \Delta)^2$, $\dom(T) = H^4(\bbR^n)$ in $L^2(\bbR^n)$, 
$ n \geq 5$, and suppose that 
\begin{equation}
\dom\big(|T|^{1/2}\big)\subseteq\dom\big(|W|^{1/2}\big) 
\end{equation} 
(representing a relative form boundedness condition). Assume  
\begin{equation} 
\sum_{j \in J} \phi_j(\cdot)^2 = 1, \quad 
\bigg\| \sum_{j \in J}  |\nabla \phi_j(\cdot)|^2 \bigg\|_{L^{\infty}(\bbR^n)} < \infty, \quad 
\bigg\| \sum_{j \in J}  |(\Delta \phi_j)(\cdot)|^2 \bigg\|_{L^{\infty}(\bbR^n)} < \infty.
\end{equation}
Then given $\varepsilon > 0$, the elementary estimate 
\begin{align*}
& \sum_{j \in J} \int_{\bbR^n} |\Delta (\phi_j f)(x)|^2 \, d^n x 
 \leq \int_{\bbR^n}  |(\Delta f)(x)|^2 \, d^n x    \\
 & \qquad + \bigg\| \sum_{j \in J}  |(\Delta \phi_j)(\cdot)|^2 \bigg\|_{L^{\infty}(\bbR^n)} \|f\|^2_{L^2(\bbR^n)}   \\
& \qquad + 4 \bigg\|\sum_{j \in J} |(\Delta \phi_j)(\cdot)| 
 |(\nabla \phi_j)(\cdot)|\bigg\|_{L^{\infty}(\bbR^n)} 
 \int_{\bbR^n} |(\nabla f)(x)| |f(x)| \, d^n x    \\
 & \qquad + 2 \bigg\| \sum_{j \in J}  |(\Delta \phi_j)(\cdot)| |\phi_j(\cdot)| \bigg\|_{L^{\infty}(\bbR^n)}  \int_{\bbR^n} |(\Delta f)(x)| |f(x)| \, d^n x   \\
 & \qquad + 4 \bigg\|\sum_{j \in J} |\phi_j(\cdot)| 
 |(\nabla \phi_j)(\cdot)|\bigg\|_{L^{\infty}(\bbR^n)} 
 \int_{\bbR^n} |(\nabla f)(x)| |(\Delta f)(x)| \, d^n x    \\
 & \quad \leq (1 + \varepsilon) \int_{\bbR^n} |(\Delta f)(x)|^2 \, d^n x 
 + C_{\varepsilon} \, \|f\|^2_{L^2(\bbR^n)}, \quad f \in H^2(\bbR^n),   
\end{align*}
for some constant $C_{\varepsilon} \in (0,\infty)$, shows that \begin{align}
& \sum_{j \in J} \big\||T|^{1/2} (\phi_j f)\big\|^2_{L^2(\bbR^n)} 
= \sum_{j \in J} \int_{\bbR^n} |\Delta (\phi_j f)(x)|^2 \, d^n x    \\
 & \quad \leq (1 + \varepsilon) \int_{\bbR^n} |(\Delta f)(x)|^2 \, d^n x 
 + C_{\varepsilon} \|f\|^2_{L^2(\bbR^n)}.
\end{align}
Thus, for arbitrary $\varepsilon > 0$, also condition $(iii)$ of
Theorem \ref{t3.1} holds with $d=1 + \varepsilon$.

Strongly singular potentials $W$ that are covered by Theorem \ref{t3.1} are, for instance, of the following form: Let $J\subseteq \bbN$ be an index set, and $\{x_j\}_{j\in J}\subset \bbR^n$, 
$n \in \bbN$, $n \geq 3$, be a set of points such that 
\begin{equation}
\inf_{\substack{j, j' \in J \\ j \neq j'}} |x_j - x_{j'}| > 0. 
\end{equation}    
Let $\phi$ be a nonnegative smooth function which equals $1$ in $B_n(0;1/2)$ and vanishes outside $B_n(0;1)$. Let $\sum_{j \in J} \phi(x-x_j)^2 \ge 1/2$, $x\in\bbR^n$, and set 
\begin{equation} 
\phi_j(x) = \phi(x-x_j) \big[\sum_{j' \in J} \phi(x-x_{j'})^2\big]^{-1/2}, \quad 
x \in \bbR^n, \; j \in J, 
\end{equation} 
such that $\sum_{j\in J} \phi_j(x)^2 =1$, $x \in \bbR^n$.
In addition, let $\gamma_j \in \bbR$, $j \in J$, $\gamma, \delta \in (0, \infty)$ with 
\begin{equation} 
|\gamma_j| \leq \gamma < \big[n (n - 4)/4\big]^2, \; j \in J,  
\end{equation} 
and consider 
\begin{equation} 
W_0(x) = \sum_{j \in J} \gamma_j \, 
|x - x_j|^{-4} \, e^{- \delta |x - x_j|}, \quad  
x \in \bbR^n \backslash \{x_j\}_{j \in J}. 
\end{equation}
Then combining Rellich's inequality in $\bbR^n$, $n \geq 5$ 
(cf.\ Corollary \ref{c2.2}) and Theorem \ref{t3.1} (with $c=1$ and $d = 1 + \varepsilon$ for arbitrary $\varepsilon > 0$), $W_0$ is form bounded 
with respect to $T= (- \Delta)^2$ with form bound strictly less than one. 

\medskip

\noindent {\bf Acknowledgments.}
We are indebted to Mark Ashbaugh and Roger Lewis for very valuable hints to the literature on Hardy--Rellich-type inequalities and to  Richard Wellman for helpful discussions.  


\end{document}